\documentclass[12pt]{article}
\usepackage{amssymb,amsfonts,amsmath,amsthm}
\usepackage{comment}
\usepackage[normalem]{ulem}
\usepackage[colorinlistoftodos,prependcaption,textsize=tiny]{todonotes}

\usepackage{color}
\usepackage{graphicx}
\newcommand{\samed}{\overset{d}{\sim}}

\newcommand{\pp}{{\mathbb P}}
\newcommand{\witi}{\widetilde}

\newcommand{\zz}{{\mathbb Z}}
\newcommand{\Z}{\zz}
\newcommand{\nn}{{\mathbb N}}
\newcommand{\N}{\nn}

\newcommand{\eqd}{{\overset{d}{\sim}}}

\newcommand{\veps}{\varepsilon}

\newcommand{\beq}{\begin{eqnarray*}}
\newcommand{\feq}{\end{eqnarray*}}
\newcommand{\beqn}{\begin{eqnarray}}
\newcommand{\feqn}{\end{eqnarray}}

\newcommand{\ind}{{\bf I}}
\newtheorem{theorem}{Theorem}
\makeatletter \@addtoreset{theorem}{section}\makeatother

\newtheorem{lemma}[theorem]{Lemma}
\newtheorem{assume}[theorem]{Assumption}

\newtheorem*{theorem*}{Theorem}

\newtheorem{proposition}[theorem]{Proposition}
\newtheorem{corollary}[theorem]{Corollary}

\begin{document}
\title{A random walk with catastrophes}
\author{
Iddo Ben-Ari\thanks{Department of Mathematics, University of Connecticut, Storrs, CT 06269-1009, USA; \newline e-mail:
iddo.ben-ari@uconn.edu}
\and
Alexander~Roitershtein\thanks{Department of Mathematics, Iowa State University, Ames, IA 50011, USA;\newline e-mail: roiterst@iastate.edu}
\and
Rinaldo B. Schinazi \thanks{Department of Mathematics, University of Colorado, Colorado Springs, CO 80933-7150, USA;\newline e-mail:
rschinaz@uccs.edu}
}
\date{September 9, 2017}
\maketitle
\begin{abstract}
	Random population dynamics with catastrophes (events pertaining to possible elimination of  a large portion of the population) has a long history in the mathematical literature. In this paper we study an ergodic model for random population dynamics with linear growth and binomial catastrophes: in a catastrophe, each individual survives with some fixed probability, independently of the rest. Through a  coupling construction,  we obtain sharp two-sided bounds for the rate of convergence to stationarity which are applied to show that the model exhibits a cutoff phenomenon.
\end{abstract}
{\em MSC2010: } Primary~60J10, 60J80, secondary~92D25, 60K37.\\
\noindent{\em Keywords}: population models, catastrophes, persistence, spectral gap, cutoff.

\section{Introduction}
\subsection{Model}
\label{intro}
Consider a population with the following birth and death rules. Given two parameters $p\in (0,1)$ and $c\in (0,1]$, the population size is a discrete-time Markov chain $(X_t:t\in \Z_+)$ on the state-space $\Z_+$ of non-negative integers ($\zz_+:=\nn\cup \{0\}$) with transition function
$$\mathfrak p^{p,c} (i,j) = \begin{cases} p & j= i+1\\
(1-p)\binom{i}{j}(1-c)^{j}c^{i-j} & i \in \{0,\dots,j\}
\end{cases}
$$
When there is no risk of ambiguity, we will omit the superscripts $p,c$ and write ${\mathfrak p}$.
In words,  conditioned on the history of the process up to time $t$, the population size  at time $t+1$  is determined by tossing an independent coin with probability $p$ of success. In the case of success, the population increases by $1$, and in the case of a failure, also known as a {\it catastrophe}, the population is an independent binomial with parameters $X_t$ and $1-c$. That is, in a catastrophe, each individual survives with probability $1-c$ independently of the other, and is otherwise killed. Note that ${\mathfrak p}$ is aperiodic and irreducible. As will be shown below, ${\mathfrak p}$ is also geometrically ergodic in total variation.
\par
The model is a version of subcritical branching process (the catastrophes) with linear migration (population increase), and belongs to a larger class of stochastic models with catastrophes  extensively studied in the literature. The term ``catastrophe" loosely refers to events where a large proportion or the entire population may be wiped out. There are many ways to model catastrophes and several were studied in the literature.  We discuss the literature in Section \ref{sec:literature} below. The particular model we study corresponds to binomial catastrophes of \cite[Section 2]{Neuts}.
\subsection{Motivation}
 Our original interest in the model came from a curiously strong persistence feature we observed in simulations: repulsion from zero and long fluctuations in a narrow band before first hitting zero. Figure~\ref{fig1} shows a simulation of the model for $p=0.4$ and $c=0.01$, between times $0$ and $10^5$. The initial population size is $X_0=10$. The population climbs quickly  and fluctuates in a narrow band  around an empirical mean close to $66.667$ for a very long time.  In Section \ref{average} we show that the process is mean-reverting around the mean of its stationary distribution. Corollary \ref{cor:bigtheta} shows that already after 1500 steps the total variation distance between the process and its stationary distribution is bounded above by $0.001$. These, along with the fact that the expectation of the first extinction time is of the order $10^{24}$, shown in Section \ref{average}, give at least a partial explanation to the simulations.
 \par
 Additional motivation for our work on the model is in its amenability to coupling methods yielding sharp bounds on the rate of convergence to stationarity. These  allow us to prove that the process exhibits the cutoff phenomenon. These results form the bulk of our work.
\begin{figure}[ht]
	\label{fig1}
	\includegraphics[width=10cm]{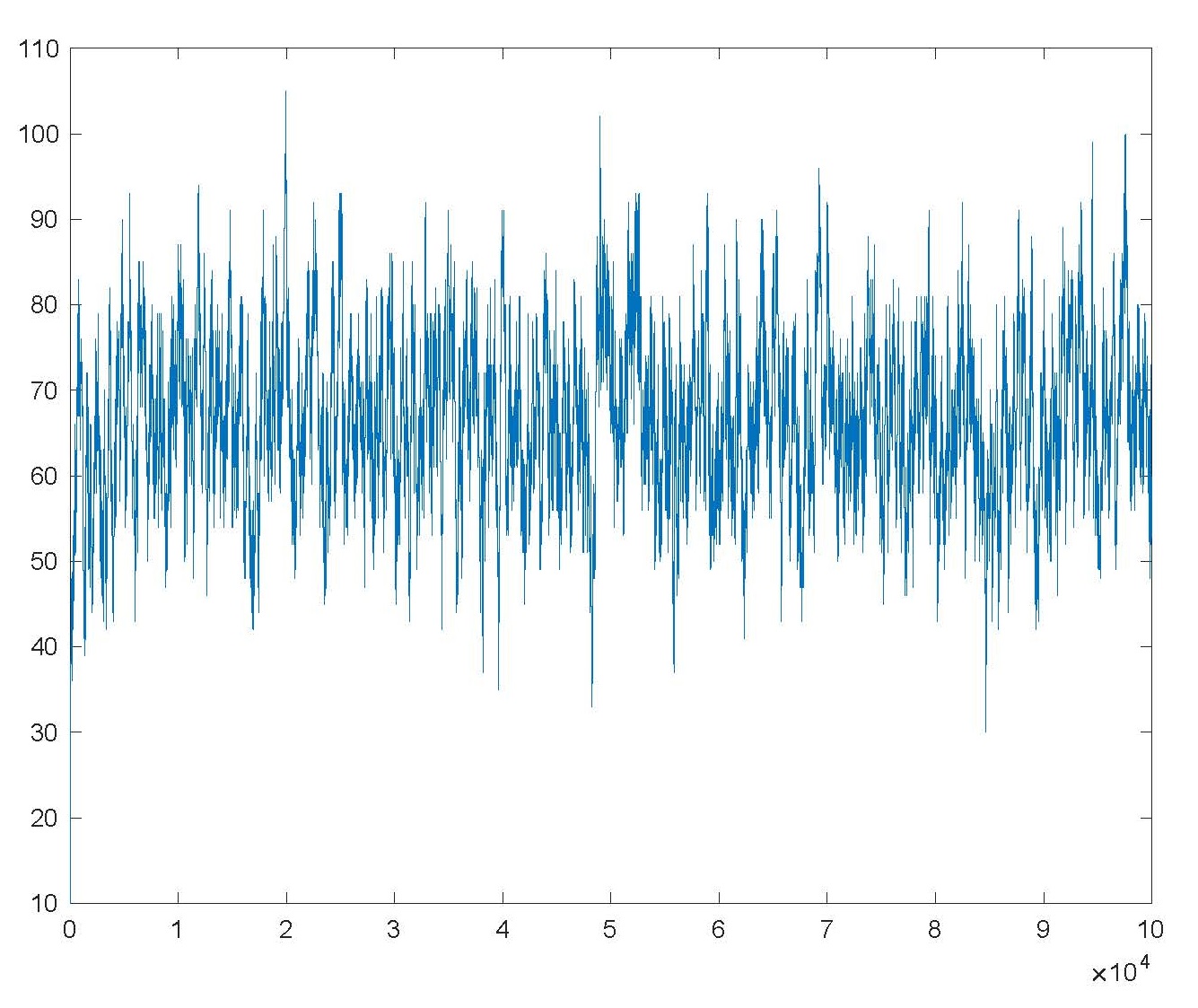}
	\caption{Long fluctuations of the random walk around its average before hitting zero}
	\end{figure}
\subsection{Literature}
\label{sec:literature}
Stochastic models with catastrophes are studied in mathematical literature since mid-1970's, for a first systematic account and a review
of the early literature see \cite{Brockwell et al.}. For a motivation and background in biological sciences see, for instance, \cite{Ewens, Hanson, Lande, Mangel}. Most of the work in the literature concern with either continuous time (generalized) birth and death chains with catastrophes or ODE-based models with a random disturbance. For a recent review and an extensive bibliography see \cite{Kapodistria}. The persistence feature is discussed for models with catastrophes in, for instance, \cite{persistent4,Mangel}. We remark that despite the variety of mathematical approaches to modeling population catastrophes, some results seem to be of a universal nature and are exhibited by models of different types. As an example, we mention the logarithmic dependence of the first  extinction time on the initial population size which we discuss in Section~\ref{subs-ext}.
\par
As mentioned above, the model we study is a particular version of binomial catastrophes case in the model introduced by Neuts in \cite[Section~2]{Neuts}. A continuous-time analogue of our model was introduced in \cite[Section~4]{Brockwell et al.}. For recent progress, see \cite{Artalejo,Economou,Kapodistria}.  In our model, deaths occur in a branching fashion, and in Section~\ref{branch} we reformulate and discuss the model as a special branching process with immigration in a random environment. The study of branching processes as  models of population growth with catastrophes (or disasters) goes back to at least \cite{Kaplan}, where a branching process without immigration is considered. Due to their tractability, much attention in the literature has been received by models with a deterministic growth between catastrophes, so called semi-stochastic models \cite{Cairns,Hanson,Hanson1,Leite}.
\par
Many results in the literature focus on the phase transition between survival and non survival, see \cite{Brockwell} and \cite{Fabio}. The results concerning first extinction times and stationary distributions are typically given in terms of Laplace transform or generating functions, see \cite{Brockwell} and  \cite{Kapodistria}.
\subsection{Organization}
 \par
 In Section 2 we give a probabilistic representation of the stationary distribution of the process.
The bulk of our contribution is reported in Sections~\ref{coupling} and \ref{poiss}. In Section~\ref{coupling} we introduce a coupling and use it compute sharp bounds on the total variation distance between the distributions of the process starting from two different initial states. In Section~\ref{poiss} we consider a sequence of models whose stationary  distribution converges to a Poisson limit. We show that this sequence exhibit a cutoff phenomenon, namely on a certain time scale the total variation distance to the stationary distribution drops from one to zero in a narrow time window.  Our study of both topics appear to be original in the context of stochastic models with catastrophes and we are not aware of similar results in the literature for any type of such models. Finally, in Section 5 we estimate the first extinction time and we use a branching representation for several purposes.
\par
Throughout the paper, the notation $a_n\sim b_n$ stands for $\lim_{n\to\infty} \frac{a_n}{b_n}=1$, and $X\eqd Y$ indicates that
the random variables $X$ and $Y$ have the same distribution.
\section{Stationary distribution}
\label{average}
\subsection{Representation formula}
\label{sec:repform}
Given a $\Z_+$-valued random variable $R$ and $\veps \in [0,1]$, write $\mbox{Bin}(R,\veps)$ for the random variable which, conditioned on $R$ is binomial with parameters $R$ and $\veps$. We begin with the following lemma  whose proof is omitted.
\begin{lemma}
\label{lem:binom}
Suppose that $R_0,R_1\dots$ are independent $\Z_+$-valued random variables and let $\veps_0,\veps_1,\veps_2,\dots$ be a sequence taking values in $[0,1]$.  Assume that $\sum \veps_j E[R_j] <\infty$. For $j=0,1,\dots$, let $\mbox{Bin}_j(R_j,\veps_j)$ be  $\mbox{Bin}(R_j,\veps_j)$-distributed, with $(\mbox{Bin}_j(R_j,\veps_j):j\ge 1)$ independent, conditional on $(R_j:j\in \Z_+)$. Let $R=\sum_{j=0}^\infty \mbox{Bin}_j(R_j,\veps_j)$ and $\veps>0$. Then $\mbox{Bin}(R,\veps)$ has the same distribution as $\sum_{j=0}^\infty \mbox{Bin}_j(R_j,\veps \veps_j)$.
\end{lemma}
For $\alpha \in (0,1]$, write $\mbox{Geom}^{-}(\alpha)$ for the shifted Geometric distribution with probability mass function equal to $(1-\alpha)^k\alpha,k\in\Z_+$. Observe that if  $R_0\samed \mbox{Geom}^{-}(\alpha)$, then
\begin{equation}
\label{eq:laplace_shifted_geom}
 E[s^{R_0}]=\alpha\sum_{k=0}^\infty s^k(1-\alpha)^k= \frac{\alpha}{1-(1-\alpha)s},\qquad s\in[0,1].
\end{equation}
The following proposition gives the stationary distribution for $X$. Note that \cite[formula (12)]{Neuts} gives the generating function of the stationary distribution for a class of Markov chains. Our model is in that class. The next proposition gives a probabilistic representation of the stationary distribution for our model .  An interpretation through branching processes representation is discussed in Section \ref{branch}.
\begin{proposition}
\label{pr:inv}
Let $R_0,R_2,\dots$ be IID $\mbox{Geom}^-(1-p)$, $\veps_j=(1-c)^j$ for $j\in \Z_+$. Let $R$ be as in Lemma~\ref{lem:binom}, and let $\pi$ be its distribution. Then $\pi$ is  stationary for $\mathfrak p$.
\end{proposition}

In the degenerate case $c=1$, $\pi$ is $\mbox{Geom}^-(1-p)$-distributed. In Section~\ref{branch} we  discuss the case when $p$ and $c$ are both close to one.

\begin{proof}[Proof of Proposition \ref{pr:inv}]
Suppose that $X_0\sim R$. We verify that $X_1\sim R$ through the generating function of $X_1.$
For $s\in[0,1],$ we have
\begin{equation}
\label{eq:stat}
E[ s^{X_1}]=p s E[s^{X_0}]+(1-p) E\bigl[s^{\mbox{\small Bin}(X_0,1-c)}\bigr].
\end{equation}
By Lemma \ref{lem:binom}, we have that
\beq
\mbox{Bin}(X_0,1-c) \eqd \sum_{j=0}^\infty \mbox{Bin}_j \bigl(R_j,(1-c)(1-c)^j\bigr)\eqd \sum_{j=1}^\infty \mbox{Bin}_j\bigl(R_j,(1-c)^j\bigr).
\feq
The sum of $\mbox{Bin}(X_0,1-c)$ and $R_0$ has the same distribution as $X_0$. In other words:
\beq
E\bigl[s^{\mbox{\small Bin}(X_0,1-c)}\bigr]\cdot E[ s^{R_0}]=E[s^{X_0}].
\feq
Thus \eqref{eq:stat} becomes
\beq
E[ s^{X_1}]=\Bigl(ps+ \frac{1-p}{E[s^{R_0}]}\Bigr)E [s^{X_0}]=E [ s^{X_0}],
\feq
where the last identity is due to \eqref{eq:laplace_shifted_geom}  with $\alpha=1-p.$
\end{proof}
Before continuing to our next topic we briefly discuss several related observations.

Using Proposition~\ref{pr:inv} and identity \eqref{eq:laplace_shifted_geom} we have
\begin{align*}
\pi(0)& =\prod_{j=0}^\infty P\bigl(\mbox{Bin}_j\bigl(R_j,(1-c)^j\bigr)=0\bigr)=
\prod_{j=0}^\infty E\bigl[\bigl(1-(1-c)^j\bigr)^{R_j}\bigr]
      \\ & =\prod_{j=0}^\infty \frac{1-p}{1-p(1-(1-c)^j)}.
\end{align*}
Let $\tau$ be the hitting time of $0$, or the first extinction time,
\begin{equation}
\label{tau}
\tau = \inf\{t\ge 1: X_t = 0\}.
\end{equation}
Thus,
\beq E_0[\tau]=\frac{1}{\pi(0)}=\prod_{j=0}^\infty \Bigl(1+\frac{p}{1-p}(1-c)^j\Bigr).
\feq
In the biological literature, this  expected value is often referred to as the \textit{persistence time} of the model \cite{Cairns,persistent,Mangel}. Using that
\beqn
\label{lineq}
x-\frac{x^2}{2}\leq \ln(1+x)\leq x, \qquad \forall~|x|<1,
\feqn
we get for $p<1/2$ that
\beq
\frac{p}{c(1-p)}-\frac{1}{2}\frac{p^2}{(1-p)^2(1-(1-c)^2)}\leq \ln E_0[\tau]\leq \frac{p}{c(1-p)}.
\feq
For example, for  $p=0.4$ and $c=0.1$, we get $E_0[\tau]\ge244$. For $p=0.4$ and $c=0.01$, we get $ E_0[\tau]\ge 10^{24}$.
\subsection{Mean Reversal}
\label{fmean}
 It follows from Proposition~\ref{pr:inv} that
\beqn
\label{exp}
\mu:=E_\pi[X_n]=\sum_{j=0}^\infty E[R_j](1-c)^j=\frac{p}{c(1-p)}.
\feqn
Note that the local drift of $X$
\beqn
\label{ld}
\delta_t:=E[X_{t+1}|X_t]-X_t=p-(1-p)cX_t=p\Bigl(1-\frac{X_t}{\mu}\Bigr)
\feqn
has the sign opposite to the deviation from $\mu$. Thus the random walk always drifts toward its expected value.  We also comment that the probability to hit $0$ in the next step decays geometrically with the state of the system, that is
\beq
P(X_{t+1}=0|X_t)=(1-p)c^{X_t}.
\feq
These observations suggest that the process will tend to fluctuate about its mean before the first extinction, as can be seen in the simulation, see Figure \ref{fig1}.
\section{Coupling and convergence to stationarity}
\label{coupling}
\subsection{Construction of the coupling}
\label{sec:coupling}
The key result of this section is a coupling of the probability laws $P_x(X_t\in\cdot)$ and $P_y(X_t\in\cdot)$, obtained from a simple representation of the process.

Let $x,y\in \Z_+$ with $x<y$. Set $X_0=x$, $X_0'=y$ and $H_0=y-x$.
We continue inductively, assuming $((X_s,X'_s,H_s),s\le t)$ were defined and $ X'_s=X_s+H_s$ for all $s\le t$.  Conditioned on $( (X_s,X'_s,H_s),s\le t)$,
 \begin{itemize}
 	\item With probability $p$, independently of the past, $H_{t+1}=H_t$, $X_{t+1}=X_t+1$ and  $X'_{t+1}=X'_t +1$.
 	\item  Otherwise, that is with probability $1-p$, set
	$$X_{t+1} = \mbox{Bin}(X_t,1-c)\mbox{ and }H_{t+1}=\mbox{Bin}(H_t,1-c),$$
	independent of each other and of the past. Moreover, set $  X'_{t+1}=X_{t+1} + H_{t+1}.$
 \end{itemize}
 It immediately follows that $X$ and $X'$ are both copies of our Markov chain and that
 $$X_t'=X_t +H_t$$
 for all $t$. In addition, the process $(H_t:t\in \Z_+)$ is non-increasing. Write $P_{x,y}$ and $E_{x,y}$ for the joint distribution and expectation of  $X$ and $X'$.  Let $\xi$ be the coupling time of two marginal processes, that is
\beqn
\label{ctime}
\xi=\inf\{t\ge 0: X_t=X'_t\}=\inf\{t\in \Z_+:H_t=0\}
\feqn
If $H_t>0$, then $H_{t+1}=H_t$ with probability equal to
$$p+(1-p)(1-c)^{H_t}<p+(1-p)(1-c)= 1-c(1-p).$$
Therefore, it immediately  follows that under $P_{x,y}$,  $\xi$ is stochastically dominated by a sum of $y-x$ independent copies of Geometric random variables with parameter $(1-p)c$. Hence, $\xi<\infty $, $P_{x,y}$-a.s. and has a geometric tail. Furthermore, $X_t=X'_t$ for all $t \ge \xi$.  Let $N_t$ denote the number of catastrophes up to time $t$.
Then  $N_t\eqd \mbox{Bin}(t,1-p)$. It follows from the construction of the coupling that
\begin{equation}
	\label{eq:H_cond}
	P_{x,y}(H_t\in \cdot | N_t )  \eqd \mbox{Bin}(y-x,(1-c)^{N_t}).
\end{equation}
Therefore,
\begin{equation}
\label{eq:zetad}
P_{x,y}(\xi > t)= P_{x,y}(H_t > 0)= 1- E [ (1-(1-c)^{N_t})^{y-x} ]\le  (y-x)E [ (1-c)^{N_t}],
\end{equation}
where the last inequality is due to Bernoulli's inequality.
Letting \begin{equation}
\label{ap}\alpha= p+(1-p)(1-c)=1-c (1-p),
\end{equation}
 we have that
\begin{equation}
\label{eq:binom_moment}
 E [(1-c)^{N_t}] =\alpha^t.
 \end{equation}
 Therefore
 \begin{equation}
 \label{eq:xi_ubound}
 P_{x,y}(\xi >t)\le (y-x)\alpha^t.
 \end{equation}
 We comment that this bound is asymptotically sharp as $t\to\infty$. That is
 \beqn
 \label{asym}
 P_{x,y}(\xi>t)\sim (y-x)E\bigl[ (1-c)^{N_t}\bigr]=(y-x)\alpha^t,
 \feqn
 as can be seen by expanding the expression $(1-(1-c)^{N_t})^{y-x}$ through the binomial theorem and taking expectation.
\subsection{Upper bounds on total variation}
\label{tot_up}
Recall that the total variation distance between two probability measures $Q_1$ and $Q_2$ on $\Z_+$ is defined as
$$\|Q_1-Q_2\|_{TV}=\max_{A\subset \zz_+} |Q_1(A)-Q_2(A)|=\max_{A\subset \Z_+}\left( Q_1(A) - Q_2(A)\right).$$
For $x,y,t\in \Z_+$, let
\beq
d_t(x,y):=\|P_y(X_t\in \cdot)-P_x(X_t\in \cdot)\|_{TV} .
\feq
By Aldous' coupling inequality \cite{Thor}, $d_t(x,y) \le P_{x,y}(\xi >t)$.
By combining this inequality and \eqref{eq:xi_ubound} we have proved
\begin{proposition}
	\label{pr:cthm}
	 Let $\alpha = 1- c (1-p)$. Then for $x,y,t \in \Z_+$,
	  $$d_t(x,y)\leq |y-x|\alpha^t.$$
\end{proposition}
Recall from \eqref{exp} that
 $$\mu = \sum_y y\pi(y)= \frac{p}{c(1-p)}.$$
 We have
\begin{corollary}
For all $x,t\in\zz_+,$
	\label{corol}
	$$ d_t(x,\pi) = \|P_x (X_t \in \cdot) - \pi \|_{TV} \le
\Bigl(x-\mu + 2 \sum_{y>x} (y-x)\pi(y)\Bigr) \alpha^t.$$
In particular,
$$ d_t (0,\pi)\le \mu \alpha^t.$$
\end{corollary}
\begin{proof}
	For any $A\subset \zz_+,$
	$$
	\left |P_x (X_t \in A) - \pi(A)\right |\leq \sum_{y=0}^\infty \left |P_x (X_t \in A) - P_y (X_t \in A)\right|\pi(y)\le \sum_{y} |y-x|\pi(y) \alpha^t,$$
	where the inequality follows from Proposition \ref{pr:cthm}. The result follows because of the identity
	$ \sum_{y} |y-x| \pi(y)=x-\mu + 2 \sum_{y>x} (y-x)\pi(y)$.
\end{proof}
\subsection{Lower bounds on total variation}
The goal of this section is to obtain a lower bound for $d_t(x,y)$ which is of the same order as the lower bound in  Proposition~\ref{pr:cthm}. We comment that the difficulty in proving such a result stems from the fact that the state space is infinite, because couplings which preserve linear ordering on a finite state space always satisfy this property, see \cite{Burdzy} for a proof in continuous-time setting. \par

We need to introduce some notation.  Let
$$ \witi p = \frac{p}\alpha = \frac{p}{1-c(1-p)}.$$
The notation $P^{(\witi p)}_x$ is the law of the Markov chain $X$ with initial state $X_0=x$, and  transition function ${\mathfrak p}^{\witi p,c}$. We will also refer to the corresponding stationary distribution as $\pi^{(\witi p)}$.

The main result of this section is the following theorem.
\begin{theorem}
\label{lthm}
Let  $x,y,t\in\Z_+$ with $x<y$. Then
$$d_t(x,y)\ge \alpha^t \max_{j\in\zz_+}\sum_{k=x}^{y-1}P_k^{(\witi p)}(X_t=j).$$
\end{theorem}
Before turning to the proof, we note the following
\begin{corollary}
	\label{corol1}
	Suppose that   $j^*\in\zz_+$ maximizes  $\pi^{(\witi p)}(\cdot)$.  Then
	\beq
	\pi^{(\witi p)}(j^*) \le  \liminf_{t\to\infty}  \frac{d_t(x,y)}{|y-x|\alpha^t} \le\limsup_{t\to\infty} \frac{d_t(x,y)}{|y-x|\alpha^t}\le 1.
	\feq
\end{corollary}

In particular, $\lim_{t\to\infty} \frac{1}{t}\log d_t(x,y)=\alpha$, thus the $L_\infty$ spectral gap of the Markov chain $X$ is
	$1-\alpha=p(1-c),$ see \cite{Kmeyn}.
 The upper bound is Proposition \ref{pr:cthm}. As for the lower bound,
 the ergodicity of the chain ${\mathfrak p}^{\witi p,c}$
 shows that for every $j\in \Z_+$, each summand in Theorem \ref{lthm}  $P_k^{(\witi p)}(X_t=j)\sim \pi^{(\witi p)}(j)$ as  $t\to\infty$. \par

 We prove Theorem \ref{lthm} through two lemmas.
\begin{lemma}~
\label{lem:onemore}
For all $x,t\in\zz_+,$
\begin{enumerate}
	\item $P_{x,x+1}(\xi >t)=\alpha^t$.
	\item $P_{x,x+1}(X_t \in \cdot | \xi >t  )= P_x^{(\witi p)}(X_t\in \cdot).$
	\end{enumerate}
\end{lemma}
\begin{proof}[Proof of Lemma~\ref{lem:onemore}]
	The first claim follows immediately from \eqref{eq:H_cond} with $y=x+1$. We turn to the second claim. Conditioned on $N_t$, $\xi$ and $X_t$ are independent. Therefore,
\beq
P_{x,x+1}(X_t = j, \xi >t |N_t=k)&=& P_{x,x+1}(X_t=j|N_t=k)P_{x,x+1}(\xi>t|N_t=k)
\\
&=&
P_x(X_t=j|N_t=k)(1-c)^k.
\feq
 Since $N_t\samed\mbox{Bin}(t,1-p),$
\beq
(1-c)^k P(N_t = k) &=& \binom{t}{k} \bigl((1-p)(1-c)\bigr)^k p^{t-k} \binom{t}{k} (\alpha-p)^kp^{t-k}
\\
&=&
\alpha^t\binom{t}{k} (1-\witi p)^k {\witi p}^{t-k}.
\feq
This gives
\beq
P_{x,x+1}(X_t = j, \xi >t )=\alpha^t \sum_{k=0}^\infty P_x(X_t=j|N_t=k)P^{(\witi p)}(N_t=k).
\feq
The distribution of $X_t$ conditioned on $N_t$ does not depend on the parameter $p$, and from this we obtain
\beq
P_{x,x+1}(X_t=j,\xi>t)= \alpha^t P^{(\witi p)}_x (X_t=j),
\feq
and the result follows.
\end{proof}
\begin{lemma}
\label{oml}
For $j\in\zz_+,$ let $A_j = \{0,\dots,j\}.$ Then
\beq
P_x (X_t \in A_j)-P_{x+1}(X_t \in A_j)\ge \alpha^t \max_{j\in\zz_+} P_x^{(\witi p)}(X_t=j).
\feq
\end{lemma}
\begin{proof}[Proof of Lemma~\ref{oml}]
Clearly,
\begin{align*}
P_x (X_t \in A_j)-P_{x+1}(X_t \in A_j) & = E_{x,x+1}\bigl[ {\bf 1}_{A_j}(X_t)-{\bf 1}_{A_j}(X_t'),\xi>t\bigr]
\\
&
=
\sum_{k=0}^{j-1} E_{x,x+1} [ {\bf 1}_{\{k\}}(X_t)-{\bf 1}_{\{k+1\}}(X_t'),\xi>t]
\\
&
\qquad
+P_{x,x+1}(X_t=j,\xi>t)-P_{x,x+1}(X_t'=0,\xi>t).
\end{align*}
Since  for $t<\xi$ we have $X'_t=X_t+1$, it follows that the expectations under the summation sign are all zero, and that the last summand is also zero. Therefore,
\beq
P_x (X_t \in A_j)-P_{x+1}(X_t \in A_j)= P_{x,x+1}(X_t=j,\xi>t)=\alpha^t P^{(\witi p)}_x(X_t=j),
\feq
where the equality follows from Lemma \ref{lem:onemore}. The proof of the lemma is complete.
\end{proof}
\begin{proof}[Proof of Theorem \ref{lthm}]
 Let  $A_j$ be as in the proof of Lemma~\ref{oml}. Then
\beq
d_t(x,y)\ge P_x (X_t \in A_j)-P_y(X_t \in A_j)=\sum_{k=x}^{y-1}\bigl( P_k (X_t \in A_j)-P_{k+1}(X_t \in A_j) \bigr),
\feq
and the theorem follows by virtue of Lemma~\ref{oml}.
\end{proof}
We conclude this section with the following generalization of Lemma \ref{lem:onemore}.
\begin{theorem}
\label{quasi}
$P_{x,y}(X_t \in \cdot | \xi>t)$ converges in distribution to $\pi^{(\witi p)}$ as $t\to\infty.$
\end{theorem}
\begin{proof}
		First,
		$$P_{x,y}(\xi>t) = \sum_{k=1}^{y-x}P_{x,y}(H_t=k).$$
		Let $M=y-x$ and $\theta = 1-c$. Then from  \eqref{eq:H_cond}, we have
		 		$$ P_{x,y}(H_t = k) = \binom{M}{k}E [ (1-\theta^{N_t})^{M-k}\theta^{N_t k}].$$
		For  $\rho <1$, $E[\rho^{N_t}]=(\rho (1-p)+p)^t$, and  it follows from the binomial formula that
		$$ P_{x,y}(H_t=k) = \binom{M}{k}(\theta^k(1-p)+p)^t(1+o(1))\mbox{ as }t\to\infty.$$
		As a result,
		$$P_{x,y}(\xi>t) = P_{x,y}(H_t=1)(1+o(1)) = (y-x)\alpha^t(1+o(1)).$$
		Next, repeating the argument in the proof of Lemma \ref{lem:onemore} we obtain
		\begin{align*}
		P_{x,y}(X_t = j , H_t = 1)&=\sum_{k=0}^\infty P(X_t=j | N_t=k)P_{x,y}(H_t=1|N_t=k)P(N_t= k)\\
		&  =\sum_{k=0}^\infty  P_{x}^{(\witi p)}(X_t=j|N_t=k)(y-x)(1-(1-c)^k)^{y-x-1}(1-c)^kP(N_t=k)\\
		& = \alpha^t (y-x)E^{(\witi p)}_{x} [{\bf 1}_{\{j\}}(X_t)(1-(1-c)^{N_t})^{y-x-1}]\\
		& = \alpha^t (y-x)P^{(\witi p)}_{x}(X_t=j)(1+o(1))\\
		\end{align*}
		with the last line follows from the binomial theorem and the bounded convergence theorem. Thus,
		$$ 	P_{x,y}(X_t = j | H_t = 1)=P_x^{(\tilde p)}(X_t=j)+o(1).$$
		Putting it all together,
		\begin{align*} P_{x,y}(X_t=j|\xi>t) &= \frac{\sum_{k=1}^{y-x}P_{x,y}(X_t=j,H_t=k)}{P_{x,y}(\xi>t)}\\
		& =\frac{P_{x,y}(X_t=j|H_t=1)P_{x,y}(H_t=1)+P_{x,y}(H_t=1)o(1)}{P_{x,y}(H_t=1)(1+o(1))}\\
		& =P_{x,y}(X_t=j|H_t=1)+o(1)\\
		& =P^{(\witi p)}_x(X_t=j)+o(1)\\
		& = \pi^{(\witi p)}(j) + o(1).
		\end{align*}
The proof of the theorem is complete.
	\end{proof}
\section{Poisson limit and a cutoff phenomenon}
\label{poiss}
In this section we let $p$ and $c$ tend to $0$.  We will work under  the following assumption
\begin{assume}
	\label{as:poisson}
For $n\in \N$ $p_n,c_n \in (0,1)$ with $p_n \to 0$ and
$$\lim_{n\to\infty}\frac{p_n }{c_n} = \beta\in (0,\infty).$$
\end{assume}
 We will use the superscript $(n)$ to denote the dependence of the total variation distance, probability, expectation, and stationary distribution of the parameters, e.g. the stationary distribution for the process with parameters $p_n$ and $c_n$ will be denoted by $\pi^{(n)}$.
\begin{theorem}
\label{limit}
Assume \ref{as:poisson}. Then $\pi^{(n)}$ converges weakly  to  $\mbox{Pois}(\beta)$ as $n\to\infty$.
\end{theorem}
The proof is a routine calculation of moment generating functions, and the proof appears at the end of the section. We note that the actual form of the limit distribution is irrelevant for our next and main result of this section, the cutoff phenomenon, although we do rely on the tightness of $(\pi^{(n)}:n\in\N)$ to prove the second claim below.
\begin{theorem}
	\label{th:cutoff}
	Let Assumption \ref{as:poisson} hold.
	Let $(y_n:n\in\Z_+)$ be a sequence of a real numbers satisfying $\lim_{n\to\infty}y_n = \infty$.  Set
	$$ t_n = \frac{\ln y_n}{c_n}.$$
Then, for every $\epsilon>0$	
	\begin{enumerate}
			\item
			$$\displaystyle\lim_{\epsilon \to 0} \limsup_{n\to\infty}\sup_{t>t_n + \frac{1}{\epsilon c_n}}d_t^{(n)}(y_n,\pi^{(n)})=0.$$
			\item		
		$$ \lim_{n\to\infty} \inf_{t<t_n - b_n }d_t^{(n)}(y_n,\pi^{(n)})=1,$$
where
			$$ b_n = (1+\epsilon)\Bigl(\frac 12 \ln y_n+ \frac{\ln\ln y_n}{c_n} \Bigr).$$
	\end{enumerate}
\end{theorem}
Therefore with a choice of parameters as in Theorem \ref{th:cutoff}, the model exhibits a cutoff at $t_n$ with window size  $O(\max (\ln y_n,\frac{\ln\ln y_n}{c_n}))$, see \cite[p. 248]{LPW}.

To prove the theorem we will use the following lemma.
\begin{lemma}
\label{cuth}
Assume the conditions for Theorem \ref{th:cutoff} hold.
For $\theta>0,$ let
\beq
\lambda_n(\theta):=\frac {\ln y_n + \theta }{c_n} \quad \mbox{\rm and} \quad
\nu_n(\theta):=\frac{\ln y_n - \ln \ln y_n-\ln \frac{p_n}{c_n}-\frac{\theta}{(\ln y_n)^{1/4}}}{-\ln(1-c_n)}
,\qquad n\in\nn.
\feq
Then
\begin{enumerate}
\item  $\lim_{\theta\to\infty}\bigl\{\limsup_{n\to\infty} \sup_{t>\lambda_n(\theta)} d^{(n)}_t(0,y_n)\bigr\}= 0.$
\item   $\lim_{\theta\to\infty}\bigl\{\liminf_{n\to\infty} \inf_{t<\nu_n} d^{(n)}_t(0,y_n)\bigr\}= 1.$
\end{enumerate}
\end{lemma}
\label{lem:bounds}
\begin{proof}
Let $\alpha_n = 1-c_n(1-p_n)$.
Recall that
	\begin{enumerate}
\item  By Proposition \ref{pr:cthm}, for any $t>\lambda_n(\theta),$
     \begin{align*}
\ln  d^{(n)}_t(0,y_n)&\leq \ln y_n+t\ln \alpha_n\leq \ln y_n+\lambda_n(\theta)\ln \alpha_n\\
     & = \ln y_n + \bigl(\ln y_n +\theta\bigr)\frac{\ln (1-c_n(1-p_n))}{c_n(1-p_n)} (1-p_n)\\
     & \leq \ln y_n - (\ln y_n +\theta) =  -\theta,
     \end{align*}
from which the first assertion of the lemma  follows.
\item  We will use the following Chernoff-Hoeffding bounds for a binomial distribution \cite{Chernoff}. If $X\samed \mbox{Bin}(m,p)$ for some $m\in\nn$ and $p\in(0,1),$ then for any $\delta\in (0,1),$
\beqn
\label{ch}
P\bigl(X>(1+\delta)pm\bigr)\leq e^{-\frac{\delta^2pm}{2}}\qquad \mbox{\rm and}\qquad
P\bigl(X<(1-\delta)pm\bigr)\leq e^{-\frac{\delta^2pm}{3}}.
\feqn
First, observe that under $P_0^{(n)}$,  $X_t$ is stochastically dominated by the number of births up to time $t$ whose distribution is $\mbox{Bin}(t,p_n)$. Let
\beq
\label{eq:gamma}
\gamma_n(\theta):=\Bigl(1+\frac{\theta}{2(\ln y_n)^{1/4}}\Bigr)p_n\nu_n(\theta).
\feq
In what follows, in order to simplify the notation, we will simply write $\nu_n$ and $\gamma_n$  instead of, respectively,
$\nu_n(\theta)$ and $\gamma_n(\theta).$
\par
By the Chernoff-Hoeffding inequality,  for any $t\leq \nu_n,$
\beq
P^{(n)}_0\bigl(X_t \ge \gamma_n \bigr)&\leq&
P\bigl(\mbox{Bin}(t,p_n) \ge \gamma_n \bigr) \leq P\bigl(\mbox{Bin}\bigl(\nu_n,p_n\bigr) \ge \gamma_n \bigr)
\\
&\leq&
\exp\Bigl( -\frac{\theta^2p_n\nu_n}{8\sqrt{\ln y_n}}\Bigr).
\feq
Therefore,
\beqn
\label{l1}
\lim_{n\to\infty} P^{(n)}_0\bigl(X_t \ge \gamma_n \bigr)=0.
\feqn
On the other hand, under  $P_{y_n}^{(n)}$, $X_t$  stochastically dominates
$\mbox{Bin}(y_n,(1-c_n)^{N_t}),$ which in turn, dominates $\mbox{Bin}\bigl(y_n,(1-c_n)^t\bigr).$
Notice that
\beqn
\label{takea}
y_n(1-c_n)^{\nu_n}= \frac{p_n}{c_n}\cdot \ln y_n \cdot e^{\frac{\theta}{(\ln y_n)^{1/4}}}.
\feqn

Thus, for $n$ large enough, we have
\beq
\frac{\gamma_n}{y_n(1-c_n)^t}&\leq& \frac{\bigl(1+\frac{\theta}{2(\ln y_n)^{1/4}}\bigr)p_n\nu_n}{y_n(1-c_n)^{\nu_n}}
=\frac{\bigl(1+\frac{\theta}{2(\ln y_n)^{1/4}}\bigr)c_n\nu_n}{(\ln y_n) e^{\frac{\theta}{(\ln y_n)^{1/4}}}}
\leq
\frac{1+\frac{\theta}{2(\ln y_n)^{1/4}}}{e^{\frac{\theta}{(\ln y_n)^{1/4}}}}
\\
&\leq&
\Bigl(1+\frac{\theta}{2(\ln y_n)^{1/4}}\Bigr)\cdot \Bigl(1-\frac{\theta}{2(\ln y_n)^{1/4}}\Bigr)
=
1-\frac{\theta^2}{4\sqrt{\ln y_n}},
\feq
where at the last but one step we used the inequality $e^{-x}\leq 1-\frac{x}{2},$ which is true for any sufficiently small $x>0,$
with $x=\frac{\theta}{(\ln y_n)^{1/4}}.$
\par
Therefore, by the Chernoff-Hoeffding inequality,  for any $t\leq \nu_n,$
\beq
P_{y_n}^{(n)}\bigl(X_t \leq \gamma_n \bigr)&\leq& P_{y_n}^{(n)}\bigl[\mbox{Bin}\bigl(y_n,(1-c_n)^t\bigr)\leq \gamma_n \bigr]
\\
&\leq &
P_{y_n}^{(n)}\bigl[\mbox{Bin}\bigl(y_n,(1-c_n)^{\nu_n}\bigr)\leq \gamma_n \bigr]
\\
&\leq&
\exp\Bigl(-\frac{ y_n(1-c_n)^{\nu_n}\theta^4}{ 48\ln y_n }\Bigr).
\feq
Hence,
$$\sup_{t<\nu_n}P_{y_n}^{(n)}\bigl(X_t \leq \gamma_n \bigr)\leq \exp\Bigl(-\frac{ y_n(1-c_n)^{\nu_n}\theta^4}{ 48\ln y_n }\Bigr).$$
It follows from \eqref{takea} that
\beq
\label{eq:what_i_needed}
\limsup_{n\to\infty} \sup_{t\leq \nu_n}P_{y_n}^{(n)}\bigl(X_t \leq \gamma_n \bigr)\leq e^{-\frac{\beta \theta^4}{48}}.
\feq
Taking in account \eqref{l1} this implies
\beq
\liminf_{n\to\infty} \inf_{t<\nu_n}d_t^{(n)}(0,y_n)\geq 1-e^{-\frac{\beta \theta^4}{48}},
\feq
from which the second claim of the lemma follows.
\end{enumerate}
\end{proof}
In order to obtain easier expressions to work with, we observe that for $\theta$ large enough, independently of $n$, we have
$$ \nu_n(\theta)\ge \frac{\ln y_n-\ln \ln y_n -\theta}{-\ln (1-c_n)}=\frac{\ln y_n-\ln \ln y_n -\theta}{c_n} \times \underset{(*)}{\underbrace{\frac{1}{1+c_n/2+c_n^2/3+\dots }}}.$$
Since $(*) = 1-\frac{c_n}{2}+O(c_n^2)$ and we have that
\begin{align*} \nu_n (\theta)&= t_n - \frac{1}{2}\ln y_n +O(c_n)\ln y_n - \frac{\ln\ln y_n}{c_n} + \frac 12 \ln\ln y_n-O(c_n)\ln\ln y_n -\frac{\theta}{c_n} (1-o(1))\\
	& = t_n -\Bigl(\frac 12 \ln y_n + \frac{\ln\ln y_n}{c_n}\Bigr)+O(c_n)\ln y_n +\frac 12 \ln \ln y_n   - \frac{\theta}{c_n}(1-o(1)),
\end{align*}
and so for every $\theta>0$ and $\epsilon>0$,
$$ \nu_n (\theta)> t_n - (1+\epsilon)\Bigl(\frac 12 \ln y_n + \frac{\ln\ln y_n}{c_n}\Bigr),$$
provided $n$ is large enough.

This leads to the following corollary. Recall that $b_n=(1+\epsilon)\bigl(\frac 12\ln y_n + \frac{\ln\ln y_n}{c_n}\bigr)$.
 \begin{corollary}~
 	\label{cor:bigtheta}
 	Under the assumptions of Theorem \ref{th:cutoff},
 	\label{cuthc}
 	\begin{enumerate}
 		\item $\displaystyle\lim_{\epsilon \to 0}\limsup_{n\to\infty}\sup_{t\ge t_n + \frac{1}{\epsilon c_n} }d^{(n)}_t(0,y_n)=0.$
 		\item For any $\epsilon >0$,  $\displaystyle\lim_{n\to\infty}\inf_{t \le t_n -b_n}d_t^{(n)}(0,y_n)=1.$
 	\end{enumerate}
 \end{corollary}
We are ready to prove Theorem \ref{th:cutoff}.
\begin{proof}[Proof of Theorem \ref{th:cutoff}]
	We begin with the first claim.  Recall that $\alpha_n = 1-c_n(1-p_n)$. Then from the triangle inequality and the Corollary \ref{corol} we obtain
 	\begin{align*}
 d^{(n)}_{t}(y_n,\pi^{(n)})&\le d^{(n)}_t(y_n,0)+d^{(n)}_t(0,\pi^{(n)})\\
  & \le d^{(n)}_t(y_n,0)+\mu_n\alpha_n^t,
 	\end{align*}
 	where $\mu_n = \sum y \pi^{(n)}(y)= \frac{p_n}{c_n (1-p_n)}$. Now $\mu_n \to \beta$, and $\ln (\alpha_n^t) = t \ln (1-c_n (1-p_n)) \le -\frac 12 c_n t$ provided $p_n\le \frac 12$. Therefore
 	$$ \lim_{n\to\infty}\sup_{t>t_n + \frac{1}{\epsilon c_n}}\mu_n \alpha_n^t =0.$$
 	The result now follow from this, combined with the first claim in Corollary \ref{cor:bigtheta}.

	We turn the second claim. Fix $\theta>0$, and recall $\nu_n(\theta)$ from Lemma \ref{cuth}. From the proof of  Lemma \ref{cuth}, it follows that for all $t< \nu_n (\theta)$, $\limsup_{n\to\infty}\sup_{t\le \nu_n(\theta)}P_{y_n}(X_t <\gamma_n(\theta))< e^{-\beta \theta^4/48}$, where $\gamma_n = \gamma_n(\theta)$ was defined in \eqref{eq:gamma}.  Since $t_n -b_n < \nu_n(\theta)$ provided $n$ is large enough, it follows that
	\begin{equation}
 \label{eq:getout}
 \lim_{n\to\infty}\sup_{t\le t_n -b_n}P_{y_n}(X_t\le \gamma_n)=0.
 \end{equation}
	By definition,  $\gamma_n\ge p_n \nu_n (\theta)\to \infty$ as $n\to\infty$, and since
	  $$d_t(y_n, \pi^{(n)})\ge P_{y_n}^{(n)}( X_t>\gamma_n) - \pi^{(n)}\bigl(\{\gamma_n,\gamma_n+1,\dots\}\bigr),$$
	   the tightness of $(\pi^{(n)}:n\in\N)$ along with \eqref{eq:getout} give
	   $$ \lim_{n\to\infty}\sup_{t\le t_n -b_n }d_t(y_n,\pi^{(n)})=1,$$
	   completing the proof.
	\end{proof}
	We conclude this section with the proof of Theorem \ref{limit}
	\begin{proof}[Proof of Theorem \ref{limit}]
		Let $Z_n$ be a random variable distributed according to $\pi^{(n)}$. By Proposition \ref{pr:inv} we can write
		$$Z_n=\sum_{j=0}^\infty B_j(G_j),$$
		where $(G_j:j\in {\mathbb Z}_+)$ are IID $\mbox{Geom}^-(p_n)$, and  $(B_j(k):j,k\in{\mathbb Z}_+)$ are independent with $B_j(k)\sim \mbox{Bin}(k,(1-c_n)^j)$, all independent of the $G_j$'s.
		
		Let $\Lambda (t) = \ln E [ e^{-t Z_n}]$. Then
		$$\Lambda (t) = \sum_{j=0}^\infty \ln E [ e^{-t B_j(G_j)}].$$
		Now
		\begin{align*} E[e^{-t B_j (G_j)}|G_j] & =  \bigl(e^{-t}(1-c_n)^j+(1-(1-c_n)^j)\bigr)^{G_j}\\
		& = \bigl(1- q_n^j (1-e^{-t} )\bigr)^{G_j}\\
		&= e^{-\gamma_{n,j}(t) G_j},
		\end{align*}
		where $q_n=1-c_n$. Therefore
		$$ E[e^{-t B_j(G_j)}]= E[e^{-\gamma_{n,j}(t) G_j}] = \sum_{k=0}^\infty (1-p_n) p_n^{k}e^{-\gamma_{n,j}(t) k}=\frac{1-p_n}{1-p_ne^{-\gamma_{n,j}(t)}}.$$
		
		Thus,
		\begin{align*}\Lambda(t) &= -\sum_{j=0}^\infty \ln \frac{1-p_n(1-(1-e^{-t})q_n^j)}{1-p_n}\\
		&= -\sum_{j=0}^\infty \ln \Bigl(1+\frac{p_n q_n^j}{1-p_n}(1-e^{-t})\Bigr).
		\end{align*}
		For $x \in (0,1)$,
		$$ 0  \le x - \ln (1+x)\le \frac{ x^2} {2}$$
		
		Therefore,
		\begin{equation}
		\label{eq:twosided_poisson}0\le \underset{(I)}{\underbrace{\sum_{j=0}^\infty \frac{p_n q_n^j}{1-p_n}(1-  e^{-t})}}  +\Lambda(t) \le \underset{(II)}{\underbrace{
				\frac {p_n^2(1-e^{-t})^2}{2(1-p_n)^2}\sum_{j=0}^\infty q_n^{2j}}}.
		\end{equation}
		Next,
		$$(I)  = \frac{p_n}{(1-p_n)c_n}(1-e^{-t})\underset{n\to\infty}{\to} \beta (1-e^{-t}),$$
		and since  $\sum_{j=0}^\infty q_n^{2j}\le \sum_{j=0}^\infty q_n^j = \frac{1}{c_n}$,
		$$(II) \le p_n \frac{p_n}{(1-p_n)^2c_n}=p_n \beta O(1)\underset{n\to\infty}{\to} 0.$$
		We have thus proved that $\lim_{n\to\infty}\Lambda(t) = -\beta (1-e^{-t})$.
	\end{proof}
\section{Additional Topics}
\subsection{Branching process representation}
\label{branch}
We adopt a scheme of Key \cite{Key} for general branching process with immigration in random environment to give a probabilistic interpretation of the particular instance of Neuts' formula \cite{Neuts}.
Using the approach of \cite{Artalejo} we compute the generating function of the extinction time in Section~\ref{subs-ext}.

The process $X$ can be thought of as a branching process with immigration in random environment.
Branching process have been used to model growth of a population subject to random catastrophes by many authors (see, for instance, a comprehensive literature review in \cite{Kapodistria}), the idea goes back to at least \cite{Kaplan} where a branching process in random environment (without immigration) was considered. In this section we use a branching representation of our process and Key's \cite{Key} representation of its stationary
distribution for several purposes. First, it yields Lemma~\ref{etau} below stating that the extinction time $\tau$ has exponential tails, next it provides an illuminating probabilistic representation of the invariant distribution $\pi$ for our process, including the extreme case of rare but nearly total catastrophes (see the discussion after Proposition~\ref{zin} and Theorem~\ref{rare} below).
\par
Let
\beqn
\label{omega}
\omega_t=\left\{
\begin{array}{lcl}
	1&\mbox{\rm if}& \mbox{\rm a birth event occurs at time}~t\\
	0&\mbox{\rm if}& \mbox{\rm a catastrophe occurs at time}~t\\
\end{array}
\right.
\feqn
We refer to the sequence $\omega:=(\omega_t)_{t\in\zz_+}$ as a \textit{random environment}. We denote the distribution of the environment by $\pp,$ the law of the process conditional on the environment by $P_\omega,$ and the corresponding expectation by $E_\omega.$
\par
The Markov process $X$ can be described using the following branching equation:
\beqn
\label{key}
X_{t+1}=\sum_{k=1}^{X_t+I_t} U_{t,i}=\sum_{k=1}^{X_t} U_{t,i}+I_t,
\feqn
where $I_t=\omega_t$ is interpreted as the number of immigrants joining the system at generation $t$ and $U_{t,i}$ as the number of progeny of $i$-th particle living at generation $t.$ Under the probability law conditional on the environment $\omega_t,$ $U_{t,i}$ are independent Bernoulli variables with parameter $c_t:=\omega_t+(1-\omega_t)(1-c)$ which are independent of $X_t:$
\beq
P_\omega(U_{t,i}=1)=c_t\qquad  \mbox{\rm and}\qquad  P_w(U_{t,i}=0)=1-c_t.
\feq
In statistical applications, this special type of branching processes with Bernoulli reproduction mechanism is often referred to as a RCINAR(1) \emph{random coefficient integer-valued autoregressive process of order one} \cite{rcinar}. In this context, \eqref{key} is written as
\beq
X_{t+1}=(1-c_t)*X_t+I_t,\qquad t\in\zz_+,
\feq
where $(1-c_t)*$ describes the action of a \textit{binomial thinning operator} \cite{McKenzie,Weis}.
\par
Stationary distribution of branching processes with immigration in a random environment, in a general (and, in fact, multi-type) setting, was studied in \cite{Key}. In particular, it follows from results in \cite{Key} that random variable $\tau$ has exponential distribution tails (in order to deduce this, one may replace $I_t$ by 1 in \eqref{key} to be able to formally use Theorem~4.2 in \cite{Key}, and then apply a stochastic dominance argument). We state it formally as
\begin{lemma}
	\label{etau}
	There exists $a,b>0$ such that $P_0(\tau>t)\leq ae^{-bt}$ for any $t\geq 0.$
\end{lemma}
We next consider a branching process obtained from $X$ by sampling at the times when catastrophes occur. This auxiliary process has a slightly simpler
structure than the underlying process $X.$ We use it below to obtain an alternative probabilistic representation of the stationary distribution of $X.$

 Let $T_0=0$ and
\beqn
\label{tin}
T_n=\inf\{k>T_{n-1}:\omega_k=0\}.
\feqn

Observe that the sequence $(T_{n}-T_{n-1}:n\ge 1)$ is an IID sequence of $\mbox{Geom}(1-p)$ random variables.
Let $Z_n=X_{T_n}$ and $Z:=(Z_n)_{n\in\zz_+}.$
\begin{proposition}
	\label{zin}
	The Markov chain $Z$ has a unique stationary distribution $Z_\infty,$ whose generating function is given by
	\beq
	E[s^{Z_\infty}]
	=\prod_{k=1}^\infty \frac{1-p}{1-p\bigl(s(1-c)^k+1-(1-c)^k\bigr)},
	\qquad s\in [0,1].
	\feq
	Thus, in the language of Proposition~\ref{pr:inv}, $Z_\infty=R-R_0=\sum_{j=1}^\infty Bin_j\bigl(R_n,(1-c)^j\bigr).$
\end{proposition}
\begin{proof}
	Considering $R_t$ as an immigration process,  $Z_t$ can be constructed as a branching process with immigration governed by the following branching identity:
	\beqn
	\label{zib}
	Z_{t+1}=\sum_{k=1}^{Z_t+R_t} V_{t,k},\qquad t\in\zz_+,
	\feqn
	where $V_{t,k}$ are IID Bernoulli random variables, independent of the immigration process and $Z_0,$ such that
	\beq
	P(V_{t,k}=1)=1-c\qquad \mbox{\rm and}\qquad P(V_{t,k}=0)=c.
	\feq
	The result thus follows from Theorem~4.2 in \cite{Key}.
\end{proof}
We remark that an auxiliary process similar to our $(Z_n)_{n\in\zz_+}$ has been used, for instance, in \cite{Economou,Kaplan} to derive
the stationary distribution for different models with catastrophes.
\par
Following the representation of the stationary distribution in \cite{Key}, one can write
\beqn
\label{minus}
Z_\infty=\lim_{t\to\infty} \sum_{k=-t}^{-1} Z_{k,0}=\sum_{k=-\infty}^{-1} Z_{k,0},
\feqn
where $Z_{k,0}\samed \mbox{Bin}_{|k|}\bigl(R_{|k|},(1-c)^{|k|}\bigr)$ is the number of descendant  alive at time zero of a ``demo" immigrant arrived at time $k<0.$ Heuristically, in this representation $Z_\infty$ is the population at time zero of a branching process that starts at minus infinity
\cite{Key}. In between two regeneration times $T_n,$ the process goes up $\mbox{Geom}^-(1-p)$ number of times. When one observe the original chain in the stationary regime, time-wise the chain is in a random place between two random times $T_n.$ This suggests (using the key renewal theorem) that the stationary distribution of the original Markov chain should be the convolution of $Z_\infty$ and an independent  $\mbox{Geom}^-(1-p)$ variable. The result is formally confirmed in Proposition~\ref{zin}.
	We conclude this section with a brief discussion of the case of ``severe but rare" catastrophes. For a biological motivation of this regime see, for instance, \cite{Huillet,Lande,rare,rare1,VE}. Specifically, a sequence of parameters $(p_n,c_n)$ such that $p_n\to 1,$ $c_n\to 1$ as $n\to\infty,$
	and $\lim_{n\to\infty} \frac{1-c_n}{1-p_n}=\beta$ for some $\beta.$ We will denote the stationary distribution for the $n$-th model, given by Proposition~\ref{pr:inv}, by $R^{(n)}.$ Observe that
	\beqn
	\label{lpi}
	E[s^R] =\prod_{k=0}^\infty \frac{1-p}{1-p\bigl(s(1-c)^k+1-(1-c)^k\bigr)}.
	\feqn
	 With this, it is not hard to verify the following result:
	\begin{theorem}
		\label{rare}
		$R^{(n)}=R_0+A_n,$ where $A_n$ is independent of $R_0$ and converges in distribution, as $n\to\infty,$ to $\mbox{Poiss}(\beta).$
	\end{theorem}
	\begin{proof}
		Recall \eqref{lpi}, and set $x_n(k):=\frac{p_n(1-c_n)^k}{1-p_n},$ $k\in\zz_+,$ $n\in\nn,$ so that
		\beq
		\ln E\bigl[s^{R^{(n)}}\bigr]=-\sum_{k=0}^\infty \ln\bigl(1+x_n(k)(1-s)\bigr),\qquad s\in[0,1].
		\feq
		To estimate the right-hand side, one can apply to $x_n(k)$ the inequality $x-\frac{x^2}{2}\leq \ln(1+x)\leq x$ which is true for all $x>0$ sufficiently small (and hence, uniformly on $k,$
		for all $x_n(k)$ with $n$ large enough). The result follows from the fact
		\beq
		\sum_{k=1}^\infty x_n(k)=\frac{p_n(1-c_n)}{(1-p_n)c_n} \to \beta, \qquad \mbox{\rm as} ~n\to\infty,
		\feq
		and
		\beq
		\sum_{k=1}^\infty \bigl(x_n(k)\bigr)^2\leq x_n(1)\cdot \sum_{k=1}^\infty x_n(k) \to 0\cdot \beta=0, \qquad \mbox{\rm as} ~n\to\infty.
		\feq
		where we took in account that $x_n(k)$ is monotone decreasing on $k.$ Thus $\ln E\bigl[s^{R^{(n)}}\bigr]$ converges, as $n\to\infty,$
		to $-\beta(1-s),$ and the proof of the theorem is complete.
	\end{proof}
	Note that in view of Proposition~\ref{zin}, $\mbox{Poiss}(\beta)$ is the limit in distribution of $Z_\infty.$ Furthermore,
	using \eqref{minus} and a similar representation for the underlying branching process $X,$ one can by virtue of the renewal theorem interpret $-R_0$ as the time of the last catastrophe before time zero and $A_n$ as the distribution of the population right after the last catastrophe in the stationary branching process $(X_t)_{t\in\zz}.$
\subsection{First Extinction Time}
\subsubsection{Overview}
In this section we discuss the following two aspects related to the first extinction time $\tau:$
\begin{itemize}
\item Asymptotic behavior of $\tau$ under large initial population.
\item  Generating function for $\tau$.
\end{itemize}
\subsubsection{Asymptotic for large population}
\label{sec:asymp}
In this section we discuss the asymptotic behavior of the first extinction time $\tau$ when the process starts from a large population.To do that we will use the coupling construction of Section \ref{sec:coupling}.
Consider the processes $X_t^{(0)}$ and $X_t^{(n)}$ with initial populations $0$ and $n$, respectively. From our coupling we know that for every $t\geq 0$ we have
$$X_t^{(n)}=X_t^{(0)}+H_t^{(n)}.$$

 Let $\tau^{(n)}$ and $\xi^{(n)}$ be the hitting time of $0$ by $X^{(n)}$ and $H^{(n)}$, respectively:
$$ \tau^{(n)} = \inf\{t\ge: X^{(n)}_t =0\},~\xi^{(n)} = \inf\{t\ge 0:H^{(n)}_t=0\}.$$
Then $\tau^{(n)},\xi^{(n)}$ are both nondecreasing.

Let $T_0=0$ and let $T_1,T_2,\dots$ be the increasing sequence of times $X^{(0)}$ visits $0$. Then clearly,
$$\tau^{(n)} = \inf\{T_k: T_k \ge \xi^{(n)}\}.$$

This is because $X^{(n)}_t=0$  if and only if $H^{(n)}_t=0$ and $X^{(0)}_t=0$. Now let $\rho^{(n)} = \tau^{(n)} - \xi^{(n)}$. Then $\rho^{(n)}$ depends on the past of the coupled system only through the size of the population $X^{(n)}_{\xi^n}$. Thus its distribution coincides with the distribution of $\tau^{(X^0_{\xi^n})}$. By ergodicity of $X^{(0)}$, and the fact that $\xi^{(n)} \nearrow  \infty$ a.s. as $n\to\infty$, it follows that $\rho^{(n)}$  converges weakly to the distribution of $\tau$, the hitting time of $0$ under $\pi$. We have proved the following:
\begin{proposition}
	\label{pr:taun_asymp}
	$\tau^{(n)} -\xi^{(n)}$ converges in distribution to $P_{\pi}(\tau \in \cdot)$ as $n\to\infty.$
\end{proposition}
It follows from \eqref{eq:zetad} that
$$ P(\xi^n \le  t )=E [( 1-(1-c)^{N_t})^n].$$

 Let $\epsilon \in (0,1/2)$, and let $A_t =\{|N_t/t -(1-p)|<\epsilon\}$. Then by the Law of Large Numbers $P(A_t)\to 1$. We have the following two-sided bounds:
\begin{align}
\label{eq:twosidedtail}
E[ (1-(1-c)^{(1-\epsilon)(1-p)t })^n,A_t ] &\le  E [(1-(1-c)^{N_t})^n] \\
\nonumber
 & \le E [ (1-(1-c)^{t(1+\epsilon)(1-p)})^n]+P(A_t^c).
\end{align}
Let
$$d_n = -\frac{\ln n}{(1-p)\ln (1-c)}.$$
If  $t\le  (1-\epsilon)d_n$, then  it follows from the second inequality in \eqref{eq:twosidedtail} that
$$ P(\xi^{(n)} \le t )\le  \bigl(1-n^{-(1-\epsilon^2)}\bigr)^n \bigl(1+o(1)\bigr)+o(1)\to 0,$$
while if  $t\ge (1+2\epsilon)d_n$, it follows from the first inequality in \eqref{eq:twosidedtail} that
$$ P(\xi^{(n)} \le t ) \ge \bigl(1-n^{-(1+2\epsilon)(1-\epsilon)}\bigr)^n \bigl(1+o(1)\bigr)\to 1.$$

Thus $\xi^{(n)} / d_n \to 1$ in probability. This, and Proposition \ref{pr:taun_asymp} give
\begin{proposition} $\tau^{(n)} / d_n \to 1$ in probability as $n\to\infty.$
\end{proposition}	
\subsubsection{Generating function}
\label{subs-ext}
For $s\in [0,1],$ let $a_n(s)=E_n[s^\tau]$ and $\psi(s,z)=\sum_{n=1}^\infty a_nz^n.$ Note that $a_0=1.$ The process has the following first-step decomposition:
\beqn
\label{mrecur}
X_{t+1}=\ind_{\{\omega_t=1\}}(X_t+1)+\ind_{\{\omega_t=0\}}\cdot \mbox{\rm Bin}(X_t,1-c),
\feqn
where $\ind_A$ stands for the indicator of the event $A,$ namely $\ind_A(\omega)=1$ if $\omega\in A$ and $\ind_A(\omega)=0$ if $\omega\not\in A.$, and $\omega_t$ is defined in \eqref{omega}. The generating function $\psi(s,z)$ can be evaluated using \eqref{mrecur} and an analytical method of \cite{Artalejo}. In particular, we have
\begin{theorem}
	\label{taul}
	For $s\in [0,1],$ let $\eta_0(s)=1$ and
	\beq
	\eta_n(s)=(-1)^n (1-c)^{\frac{n(n-1)}{2}}\Bigl(\frac{(1-p)s}{1-ps}\Bigr)^n
	\prod_{k=1}^n \frac{1}{1-(1-c)^k},\qquad n\in\nn.
	\feq
	Then
	\beqn
	\label{tauf}
	E_1[s^\tau]=1+\frac{1-s}{ps}-\frac{\sum_{n=1}^\infty \eta_n(s)}
	{\sum_{n=1}^\infty \eta_n(s) \frac{ps(1-c)^n}{1-ps+ps(1-c)^n}}.
	\feqn
\end{theorem}
The proof of the theorem is similar to the proof of Theorem~3.1, part (ii), in \cite{Artalejo}. Namely, an application of \eqref{mrecur} leads
to a recursive equation  for the generating function $\psi(s,z)$ of a type that has been analyzed in \cite{Artalejo}. We comment that through the recurrence relation \eqref{eq:recur_an}, we can obtain an explicit formula for $E_n [s^{\tau}]$ for each $n\in\N$.
The proof below is provided for the sake of completeness.
\begin{proof}
We assume throughout the argument that $s,z\in(0,1).$ For simplicity of notation, we will occasionally suppress the dependence of underlying functions on the parameter  $s.$ Using \eqref{mrecur}, we obtain
\beq
\label{eq:recur_an}
a_n=psa_{n+1}+(1-p)s\sum_{k=0}^n \binom{n}{k}c^{n-k}(1-c)^ka_k,\qquad n\in\nn.
\feq
Multiplying by $z^n$ and summing over $n$ from $1$ to $\infty$ yields
\beq
\psi(z)-1&=&\frac{ps}{z}\bigl\{\psi(z)-1-a_1z\bigr\}-(1-p)s
\\
&&
\qquad
+(1-p)s\sum_{k=0}^\infty z^k(1-c)^ka_k
\sum_{n=k}^\infty \binom{n}{k} (cz)^{n-k}
\\
&=&
\frac{ps}{z}\bigl\{\psi(z)-1-a_1z\bigr\}-(1-p)s+\frac{(1-p)s}{1-cz}
\sum_{k=0}^\infty \Bigl(\frac{z-cz}{1-cz}\Bigr)^ka_k
\\
&=&
\frac{ps}{z}\bigl\{\psi(z)-1-a_1z\bigr\}-(1-p)s+\frac{(1-p)s}{1-cz}\psi\Bigl(\frac{z-cz}{1-cz}\Bigr),
\feq
where we used the negative binomial formula $\sum_{n=k}^\infty \binom{n}{k} x^{n-k}=(1-x)^{-k-1}$ with $x=cz.$  Thus
\beqn
\nonumber
\psi(s,z)&=&\frac{ps}{ps-z}-\frac{z}{ps-z}\bigl(1-psa_1(s)-(1-p)s\bigr)
\\
\nonumber
&&
\qquad
-\frac{(1-p)sz}{(1-cz)(ps-z)}\psi\Bigl(s,\frac{z-cz}{1-cz}\Bigr)
\\
\nonumber
&=&
1+\frac{z}{ps-z}\bigl(psa_1(s)+(1-p)s\bigr)
\\
\label{step}
&&
\qquad
-\frac{(1-p)sz}{(1-cz)(ps-z)}\psi\Bigl(s,\frac{z-cz}{1-cz}\Bigr).
\feqn
Let
\beqn
\label{gfor}
g(s,z):=ps-z+z\bigl(psa_1(s)+(1-p)s\bigr)
\feqn
and
\beq
\varphi(z)=\psi(s,z)(ps-z).
\feq
For $k\geq 1,$ let $h(z)=\frac{z-cz}{1-cz},$ $h_0(z)=z,$ and $h_k(z)=h\bigl(h_{k-1}(z)\bigr)$ for $k\in\nn.$
It is easy to verify that
\beqn
\label{hfor}
h_k(z)=\frac{z(1-c)^k}{1-\bigl(1-(1-c)^k\bigr)z}.
\feqn
In this notation, \eqref{step} can be rewritten as
\beqn
\label{it}
\varphi(s,z)=g(s,z)+\frac{(1-p)sh(z)}{(1-c)(h(z)-ps)}\varphi\bigl(s,h(z)\bigr).
\feqn
Note that $h_k(z)\in (0,z)$ for $z\in(0,1).$ Consequently, taking in account \eqref{hfor} and that $a_n(s)\in (0,1)$ for all $s\in (0,1),$
\begin{itemize}
\item [(i)] For any $z\in (0,1),$ $h_k(z)$ decreases, as $k\to\infty,$ to zero, which is the smallest of two fixed points of $h.$
\item [(ii)] $\psi\bigl(s,h_k(z)\bigr)\leq \psi(s,z)\leq \sum_{n=0}^\infty z^n <\infty$ for all $k\in\zz_+.$
\item [(iii)] We have:
\beq
-1<-z(1-s)<ps-z+z(1-p)s\leq g(s,z) \leq ps-z+z\bigl(ps+(1-p)s\bigr)<1,
\feq
and hence $g(s,z)$ is uniformly bounded for $s,z\in(0,1).$
\item [(iv)] For $z\leq ps$ and $k\in\zz_+,$
\beqn
\label{sois}
0\leq \frac{(1-p)sh_k(z)}{(1-c)\bigl(ps-h_{k}(z)\bigr)}\to 0,\qquad \mbox{\rm as}~k\to\infty.
\feqn
\end{itemize}
Thus, one can iterate \eqref{it} to obtain
\beq
\varphi(s,z)=g(s,z)+\sum_{n=1}^\infty g\bigl(s,h_n(z)\bigr)\prod_{k=1}^n \frac{(1-p)sh_k(z)}{(1-c)\bigl(h_{k}(z)-ps\bigr)}.
\feq
Plugging in into this formula $z=ps$ yields, taking into account that $\varphi(s,z)=0,$
\beqn
\label{lasteq}
0=g(s,z)+\sum_{n=1}^\infty g\bigl(s,h_n(z)\bigr)\frac{(1-p)^ns^n(1-c)^{\frac{n(n-1)}{2}}}
{(ps-1)^n}\prod_{k=1}^n \frac{1}{1-(1-c)^k}.
\feqn
This yields \eqref{tauf} by virtue of \eqref{gfor} and \eqref{hfor}. In fact, after a suitable renaming of variables,
equation \eqref{lasteq} for $a_1(s)$ is analogous to (3.12) in \cite{Artalejo}, while our \eqref{tauf} is its solution (3.4) in \cite{Artalejo}.
\end{proof}
\bibliographystyle{amsplain}

\end{document}